\documentclass[a4paper]{amsart}
\usepackage{graphicx}
\usepackage{amssymb}
\usepackage{amsmath}
\usepackage{amsthm}
\usepackage{amscd}
\usepackage[all,2cell]{xy}

\UseAllTwocells \SilentMatrices
\newtheorem{thm}{Theorem}[section]

\newtheorem{cor}[thm]{Corollary}
\newtheorem{lem}[thm]{Lemma}

\newtheorem{prop}[thm]{Proposition}
\theoremstyle{definition}

\theoremstyle{remark}
\newtheorem{rem}[thm]{\bf Remark}
\numberwithin{equation}{section}

\begin{document}
\title[Recollements, comma categories and morphic enhancements]{Recollements, comma categories and morphic enhancements}
\author[Xiao-Wu Chen, Jue Le] {Xiao-Wu Chen, Jue Le$^*$}

\date{\today}

\thanks{$^*$ The corresponding author}

\thanks{E-mail: xwchen$\symbol{64}$mail.ustc.edu.cn, juele$\symbol{64}$ustc.edu.cn}

\maketitle

\dedicatory{}%
\commby{}

\begin{abstract}
For each recollement of triangulated categories, there is an epivalence between the middle category and the comma category associated to a triangle functor from the category on the right to the category on the left. For a morphic enhancement of a triangulated category $\mathcal{T}$, there are three explicit ideals of the enhancing category,  whose corresponding factor categories are all equivalent to the module category over $\mathcal{T}$.
Examples related to inflation categories and weighted projective lines are discussed.
\end{abstract}

\section{Introduction}

The notion of a recollement is introduced in \cite{BBD}, which serves as a categorical analogue to the gluing of geometric objects. Roughly speaking, a recollement is given by six triangle functors among three triangulated categories,
\begin{align}\label{rec0}
\xymatrix{
\mathcal{T}' \ar[rr]|{i} &&  \mathcal{T} \ar[rr]|{j}  \ar@/_1pc/[ll]|{i_\lambda} \ar@/^1pc/[ll]|{i_\rho} &&  \mathcal{T}''   \ar@/_1pc/[ll]|{j_\lambda} \ar@/^1pc/[ll]|{j_\rho}
}
\end{align}
where we view $\mathcal{T}$ as glued from $\mathcal{T}'$ and $\mathcal{T}''$.

On the other hand, forming the comma category along a given functor is a standard way to glue two categories. For a functor $F\colon \mathcal{C}\rightarrow \mathcal{D}$, the comma category $(F\downarrow \mathcal{D})$ is defined such that its objects are given by  morphisms $f\colon F(C)\rightarrow D$ in $\mathcal{D}$. If $\mathcal{C}=\mathcal{D}$ and $F$ is the identity functor, the comma category coincides with the morphism category ${\rm mor}(\mathcal{C})$ of $\mathcal{C}$. In (\ref{rec0}), we consider the functor $i_\rho j_\lambda \colon \mathcal{T}''\rightarrow \mathcal{T}'$ and form the comma category $(i_\rho j_\lambda \downarrow \mathcal{T}')$.

The first result compares the above two ways of gluing. Recall that a functor $F\colon \mathcal{C}\rightarrow \mathcal{D}$ is an \emph{epivalence},  provided that it is full and dense, and detects isomorphisms between objects. Consequently, there is a bijection between the  sets of isomorphism classes of objects in $\mathcal{C}$ and $\mathcal{D}$.

\vskip 5pt
\noindent {\bf Theorem A}. \quad \emph{For any recollement (\ref{rec0}), there is an epivalence
$$\Phi\colon \mathcal{T}\longrightarrow (i_\rho j_\lambda \downarrow \mathcal{T}').$$
}

We mention that Theorem A is inspired by the characterization theorem for a morphic enhancement. Recall from \cite[Appendix C]{Kr} that a triangle functor $i\colon \mathcal{T}'\rightarrow \mathcal{T}$ is  a \emph{morphic enhancement} of $\mathcal{T}'$, provided that it fits into a recollement (\ref{rec0}) such that $\mathcal{T}''=\mathcal{T}'$ and $(i_\rho, j_\lambda)$ is an adjoint pair. We mention that the definition given here is equivalent to the original \cite[Definition C.2]{Kr}; see the proof of \cite[Theorem C.1 and Proposition C.3 a)]{Kr}.

For the morphic enhancement $i\colon \mathcal{T}'\rightarrow \mathcal{T}$, we identify $i_\rho j_\lambda$ with ${\rm Id}_{\mathcal{T}'}$ via the counit of $(i_\rho, j_\lambda)$. Therefore, the epivalence in Theorem A takes the form
$$\Phi\colon \mathcal{T}\longrightarrow {\rm mor}(\mathcal{T}').$$
This recovers the implication ``$(2)\Rightarrow (1)$" of \cite[Theorem C.1]{Kr}. We apply Theorem A to  the standard recollement for a one-point extension, and recover \cite[Appendix, Proposition 1]{Gei}.

We mention that morphic enhancements appear as the first level in  towers of triangulated categories; see \cite{Ke91}. We point out a related result \cite[Proposition~4.10]{KL} on the gluing of pretriangulated dg categories, although Theorem A and its proof are completely different.

Denote by ${\rm mod}\mbox{-}\mathcal{T}'$ the module category over $\mathcal{T}'$, that is, the category of finitely presented contravariant functors from $\mathcal{T}'$ to the category of abelian groups. The following well-known functor
$${\rm Cok}\colon {\rm mor}(\mathcal{T}')\longrightarrow {\rm mod}\mbox{-}{\mathcal{T}'}$$
sends a morphism $f\colon X\rightarrow Y$ to the cokernel functor ${\rm Cok}(\mathcal{T}'(-, f)\colon \mathcal{T}'(-, X)\rightarrow \mathcal{T}'(-, Y))$.

The second result realizes, for each morphic enhancement $i\colon \mathcal{T}'\rightarrow \mathcal{T}$,  the module category ${\rm mod}\mbox{-}\mathcal{T}'$ as a factor category of the enhancing category $\mathcal{T}$ by an idempotent ideal.

\vskip 5pt

\noindent {\bf Theorem B.}\quad \emph{Let $i\colon \mathcal{T}'\rightarrow \mathcal{T}$ be a morphic enhancement, which fits into the recollement (\ref{rec0}). Then the composite functor ${\rm Cok}\circ \Phi$ induces an equivalence
$$\mathcal{T}/{({\rm Im}\; j_\lambda+{\rm Im}\; j_\rho)} \stackrel{\sim}\longrightarrow {\rm mod}\mbox{-}{\mathcal{T}'}.$$
}

Here, ${\rm Im}\; j_\lambda$ and ${\rm Im}\; j_\rho$ denote the essential images of $j_\lambda$ and $j_\rho$, respectively. In the above equivalence, the left hand side denotes the factor category of $\mathcal{T}$ modulo the idempotent ideal generated by the full subcategory ${\rm Im}\; j_\lambda+{\rm Im}\; j_\rho$.

In the situation of Theorem B, both the functors $j_\lambda$ and $j_\rho$ are morphic enhancements of $\mathcal{T}'$. Applying Theorem B to them, we actually obtain three functors
$$\mathcal{T}\longrightarrow {\rm mod}\mbox{-}{\mathcal{T}'},$$
each of which induces an equivalence between ${\rm mod}\mbox{-}{\mathcal{T}'}$ and  a factor category of $\mathcal{T}$ by  a certain full subcategory.

The above three functors recover \cite[Corollary 1.3]{Lin} in a uniform way, which relate the stable inflation category $\underline{\rm inf}(\mathcal{A})$ of a Frobenius exact category $\mathcal{A}$ to the module category over the stable category of $\underline{\mathcal{A}}$. Indeed, the canonical functor $\underline{\mathcal{A}}\rightarrow \underline{\rm inf}(\mathcal{A})$, sending any object to its identity endomorphism, is a morphic enhancement; see \cite{Ke90}. We mention that \cite[Corollary 1.3]{Lin} generalizes and extends the main results of \cite{RZ} and \cite{Eir}, which relate, for a selfinjective algebra $A$ of finite representation type, the stable submodule category $\underline{\mathcal{S}}(A)$ over $A$ to the module category ${\rm mod}\mbox{-}\Lambda$ over the stable Auslander algebra $\Lambda$.

The category $\mathcal{T}$ is triangulated, but the module category ${\rm mod}\mbox{-}\mathcal{T}'$ is abelian. Therefore, Theorem B yields new examples for the general phenomenon: a certain factor category of a triangulated category might be an abelian category; compare \cite[Section 8]{RZ}.

The structure of the paper is straightforward. In Section 2, we recall basic facts on recollements. We study the intertwining isomorphism and Mayer-Vietoris triangles. We prove Theorem A (= Theorem~\ref{thm:A}) in Section 3 and Theorem B (= Theorem~\ref{thm:B}) in Section 4.

We apply Theorem B to various inflation categories in Section 5. In Proposition~\ref{prop:wpl}, we relate,  for each $p\geq 2$,  the stable vector bundle category $\underline{\rm vect}\mbox{-}\mathbb{X}(2,3, p)$ on the weighted projective line of weight type $(2, 3, p)$ to the $\mathbb{Z}$-graded preprojective algebra $\Pi_{p-1}$ of type $\mathbb{A}_{p-1}$.  By combining the work \cite{KLM2} and \cite{RZ}, we obtain an equivalence
$${\underline{\rm vect}\mbox{-}\mathbb{X}(2,3, p)/\mathcal{V}_2} \stackrel{\sim}\longrightarrow \underline{\rm mod}^\mathbb{Z}\mbox{-}\Pi_{p-1}$$
between the factor category of $\underline{\rm vect}\mbox{-}\mathbb{X}(2,3, p)$  by the full subcategory $\mathcal{V}_2$ consisting of all vector bundles of rank two,  and the stable graded module category over $\Pi_{p-1}$.

\section{The intertwining isomorphism and Mayer-Vietoris triangles}

In this section, we recall from \cite[1.4]{BBD} some basic notions on recollements. We study the intertwining isomorphism and Mayer-Vietoris triangles.

We fix the notation and convention. For an additive functor $F\colon \mathcal{C}\rightarrow \mathcal{D}$ between additive categories, we denote by ${\rm Ker}\; F$ the full subcategory of $\mathcal{C}$ formed by those objects annihilated by $F$. Denote by ${\rm Im}\; F$ the essential image of $F$, which is defined to be the full subcategory of $\mathcal{D}$ formed by those objects isomorphic to $F(C)$ for some object $C\in \mathcal{C}$.

Let $\mathcal{T}$ and $\mathcal{T}'$ be two triangulated categories with the translation functors $\Sigma$ and $\Sigma'$, respectively. Recall that a triangle functor $(F, \omega)\colon \mathcal{T}\rightarrow \mathcal{T}'$ consists of an additive functor $F\colon \mathcal{T}\rightarrow \mathcal{T}'$ and a natural isomorphism $\omega\colon F\Sigma \rightarrow \Sigma' F$, called the connecting isomorphism,  such that each exact triangle $X \rightarrow Y \rightarrow Z \stackrel{h}\rightarrow \Sigma(X)$ in $\mathcal{T}$ is sent to an exact triangle $F(X)\rightarrow F(Y) \rightarrow F(Z) \xrightarrow{\omega_X \circ F(f)} \Sigma' F(X)$ in $\mathcal{T}'$. In the sequel, we will suppress the connecting isomorphism $\omega$, and identify $F\Sigma$ with $\Sigma'F$. Moreover, we will denote the translation functor of any triangulated category by the same symbol $\Sigma$.

\subsection{The intertwining isomorphism}

By an \emph{exact sequence} of triangulated categories, we mean a diagram of triangle functors
$$\mathcal{T}' \stackrel{i} \longrightarrow \mathcal{T} \stackrel{j}\longrightarrow \mathcal{T}''$$
such that $i$ is fully faithful with ${\rm Im}\; i={\rm Ker}\; j$ and that $j$ induces an equivalence $\mathcal{T}/{{\rm Ker}\; j}\simeq \mathcal{T}''$. Here, $\mathcal{T}/{{\rm Ker}\; j}$ denotes the Verdier quotient of $\mathcal{T}$ by ${{\rm Ker}\; j}$.

Recall from \cite[1.4]{BBD} that a \emph{recollement} $(\mathcal{T}', \mathcal{T}, \mathcal{T}'')$ consists of three triangulated categories $\mathcal{T}'$, $\mathcal{T}$ and $\mathcal{T}''$ connected with six triangle functors,
\begin{align}\label{rec1}
\xymatrix{
\mathcal{T}' \ar[rr]|{i} &&  \mathcal{T} \ar[rr]|{j}  \ar@/_1pc/[ll]|{i_\lambda} \ar@/^1pc/[ll]|{i_\rho} &&  \mathcal{T}''   \ar@/_1pc/[ll]|{j_\lambda} \ar@/^1pc/[ll]|{j_\rho}
}
\end{align}
which are subject to the following conditions.
\begin{enumerate}
\item $(i_\lambda, i)$, $(i, i_\rho)$, $(j_\lambda, j)$ and $(j, j_\rho)$ are adjoint pairs.
\item The functors $i$, $j_\lambda$ and $j_\rho$ are fully faithful.
\item The composition $ji\simeq 0$. Consequently, we have $i_\lambda j_\lambda\simeq 0\simeq i_\rho j_\rho$.
\item For each $X\in \mathcal{T}$, there are exact triangles
\begin{align}\label{equ:rec1}
ii_\rho(X) \xrightarrow{\varepsilon_X} X \xrightarrow{\eta_X} j_\rho j(X) \xrightarrow{\delta_X} \Sigma ii_\rho(X)\end{align}
and
\begin{align}\label{equ:rec2}
j_\lambda j(X) \xrightarrow{\phi_X} X \xrightarrow{\psi_X} ii_\lambda(X) \xrightarrow{\sigma_X} \Sigma j_\lambda j(X).
\end{align}
\end{enumerate}

In the recollement above, each level gives rise to an exact sequence of triangulated categories. In particular, we have ${\rm Im}\; i={\rm Ker}\; j$, and  identify $j$ as the Verdier quotient functor $\mathcal{T}\rightarrow \mathcal{T}/{{\rm Ker}\; j}$.

The above exact triangles (\ref{equ:rec1}) and (\ref{equ:rec2})  are functorial in $X$. Therefore, they might be written as follows:
$$ii_\rho \stackrel{\varepsilon}\longrightarrow  {\rm Id}_{\mathcal{T}} \stackrel{\eta}\longrightarrow j_\rho j \stackrel{\delta}\longrightarrow \Sigma ii_\rho,$$
and
$$ j_\lambda j \stackrel{\phi}\longrightarrow {\rm Id}_\mathcal{T} \stackrel{\psi}\longrightarrow ii_\lambda \stackrel{\sigma}\longrightarrow \Sigma j_\lambda j.$$
We will call such triangles \emph{functorial} exact triangles.

We recall that  $\varepsilon$ and $\phi$ are the counits of the adjoint pairs $(i, i_\rho)$ and $(j_\lambda, j)$, respectively, and that $\eta$ and $\psi$ are the units of $(j, j_\rho)$ and $(i_\lambda, i)$, respectively. In the above exact triangles, the morphisms $\delta$ and $\sigma$ are uniquely determined by these counits and units.  By convention, we will assume that $\Sigma$ commutes with these six triangle functors. However, we notice that $\delta\Sigma=-\Sigma \delta$ and $\sigma \Sigma=-\Sigma \sigma$. Recall that $\Sigma$ is a triangle functor with the connecting isomorphism $-{\rm Id}_{\Sigma^2}$. Therefore, both $\delta$ and $\sigma$ are natural transformations between triangle functors.

The following convention will be used. We denote by $\varepsilon'\colon {\rm Id}_{\mathcal{T}'}\rightarrow i_\rho i$ and $\phi'\colon {\rm Id}_{\mathcal{T}''}\rightarrow j j_\lambda$ the units of $(i, i_\rho)$ and $(j_\lambda, j)$, respectively. Similarly, we denote by $\eta'\colon jj_\rho \rightarrow {\rm Id}_{\mathcal{T}''}$ and $\psi'\colon i_\lambda i\rightarrow {\rm Id}_{\mathcal{T}'}$ the counits of $(j, j_\rho)$ and $(i_\lambda, i)$, respectively. They are all natural isomorphisms, and commute with $\Sigma$.

Consider the following composition
$$j_\lambda j \stackrel{\phi} \longrightarrow {\rm Id}_\mathcal{T} \stackrel{\eta}\longrightarrow j_\rho j.$$
Since $j$ is identified with the Verdier quotient functor, there is a unique natural transformation
$$N\colon j_\lambda \longrightarrow j_\rho,$$
called the \emph{norm morphism}, satisfying $Nj=\eta\circ \phi$; compare \cite[1.4.6 b)]{BBD} and \cite[4.1]{FP}.

\begin{lem}\label{lem:N}
We have $N=(j_\rho \phi')^{-1}\circ \eta j_\lambda=\phi j_\rho \circ (j_\lambda \eta')^{-1}.$
\end{lem}

\begin{proof}
We only prove the first equality. Since $j$ is identified with the Verdier quotient functor, it suffices to show the following identity
$$Nj=(j_\rho \phi'j)^{-1}\circ \eta j_\lambda j.$$
 By the adjoint pair $(j_\lambda, j)$, we have $(\phi' j)^{-1}=j\phi$. Therefore, the right hand side equals $j_\rho j\phi \circ \eta j_\lambda j$, which further equals $\eta\circ \phi$ by the naturalness of $\eta$. This completes the proof.
\end{proof}

The following result is partially due to \cite[1.4.6 c)]{BBD}.

\begin{prop}\label{prop:N-xi}
In the following diagram, the two rows are functorial exact triangles.
 \begin{align}\label{equ:N-xi}\xymatrix{
 \Sigma^{-1} ii_\lambda j_\rho \ar@{.>}[d]_{\Sigma^{-1} i\xi} \ar[rr]^-{-j_\lambda \eta'\circ \Sigma^{-1}\sigma j_\rho} &&  j_\lambda \ar@{=}[d] \ar[r]^-{N} & j_\rho \ar@{=}[d] \ar[rr]^-{\psi j_\rho} && ii_\lambda j_\rho \ar@{.>}[d]^-{i\xi}\\
 ii_\rho j_\lambda \ar[rr]^-{\varepsilon j_\lambda} && j_\lambda \ar[r]^-{N} & j_\rho \ar[rr]^-{\delta j_\lambda \circ j_\rho \phi'} && \Sigma ii_\rho j_\lambda
 }\end{align}
 Then there is a unique natural isomorphism
 $$\xi\colon i_\lambda j_\rho  \stackrel{\sim}\longrightarrow \Sigma i_\rho j_\lambda,$$
 which makes the diagram  commute. Moreover, we have
 $$\xi=\Sigma \psi' i_\rho j_\lambda \circ i_\lambda \delta j_\lambda \circ i_\lambda j_\rho \phi'= - \Sigma i_\rho j_\lambda \eta'\circ i_\rho \sigma j_\rho \circ \varepsilon' i_\lambda j_\rho.$$
\end{prop}

We will refer to $\xi$  as the \emph{intertwining isomorphism} for the recollement (\ref{rec1}).

\begin{proof}
We apply (\ref{equ:rec2}) to $j_\rho(Y)$ for $Y\in \mathcal{T}''$, and identify $j_\lambda jj_\rho(Y)$ with $j_\lambda(Y)$ via the isomorphism $j_\lambda(\eta'_Y)$. In view of Lemma \ref{lem:N}, the resulting triangle can be rotated to the upper row of the above diagram. Similarly, the lower row is obtained by applying (\ref{equ:rec1}) to $j_\lambda(Y)$ and the isomorphism $j_\rho(\phi'_Y)$.

Recall that $i$ is fully faithful. By (TR3), we deduce the isomorphism $\xi_Y$ making the diagram commute. It is standard to deduce the uniqueness of $\xi_Y$ from the fact that ${\rm Hom}_\mathcal{T}(\Sigma j_\lambda(Y), \Sigma ii_\rho j_\lambda(Y))=0$. The uniqueness of $\xi_Y$ also implies its  naturalness.

For the last identity, we only prove the first equality. Indeed, we have
\begin{align*}
\xi & =\xi\circ \psi'i_\lambda j_\rho \circ i_\lambda \psi j_\rho \\
   & = \Sigma \psi' i_\rho j_\lambda \circ i_\lambda i \xi \circ i_\lambda \psi j_\rho\\
   &= \Sigma \psi' i_\rho j_\lambda \circ i_\lambda \delta  j_\lambda \circ i_\lambda j_\rho \phi'.
\end{align*}
Here, the first equality uses the fact that $\psi' i_\lambda \circ i_\lambda \psi={\rm Id}_{i_\lambda}$, the second one uses the naturalness of $\psi'$ and the last one uses the rightmost square in (\ref{equ:N-xi}).
\end{proof}

\begin{rem}\label{rem:1}
By the argument in the second paragraph of the above proof, the isomorphism $\xi$ is already unique if it is required to make the rightmost square in (\ref{equ:N-xi}) commute.
\end{rem}

Consider the following composition
$$ii_\rho \stackrel{\varepsilon}\longrightarrow {\rm Id}_\mathcal{T} \stackrel{\psi} \longrightarrow ii_\lambda.$$
Since $i$ is fully faithful, there is a unique natural transformation
$$C\colon i_\rho \longrightarrow i_\lambda,$$
called the \emph{conorm morphism}, satisfying $iC=\psi\circ \varepsilon$; compare \cite[1.4.6 a)]{BBD}.

The following fact is similar to Lemma \ref{lem:N}. We omit the similar proof.

\begin{lem}\label{lem:C}
We have $C=(\varepsilon' i_\lambda)^{-1}\circ i_\rho \psi=i_\lambda \varepsilon \circ (\psi' i_\rho)^{-1}.$
\end{lem}

The following result is analogous to Proposition \ref{prop:N-xi}. We recall the intertwining isomorphism $\xi$ therein.

\begin{prop}\label{prop:C-xi}
The following diagram is commutative,
\begin{align}\label{equ:C-xi}
\xymatrix{
\Sigma^{-1} i_\lambda j_\rho j \ar[d]_{-\Sigma^{-1}\xi j}\ar[rr]^-{-\psi' i_\rho \circ \Sigma^{-1}i_\lambda \delta} && i_\rho \ar@{=}[d]\ar[r]^-{C} & i_\lambda  \ar@{=}[d]  \ar[rr]^-{i_\lambda \eta} && i_\lambda j_\rho j \ar[d]^-{-\xi j}\\
i_\rho j_\lambda j \ar[rr]^-{i_\rho \phi} && i_\rho \ar[r]^-{C} & i_\lambda \ar[rr]^-{i_\rho \sigma \circ \varepsilon' i_\lambda} && \Sigma i_\rho j_\lambda j
}\end{align}
and its rows are functorial exact triangles.
\end{prop}

\begin{proof}
We apply $i_\lambda$ to (\ref{equ:rec1}),  and identify $i_\lambda i i_\rho(X)$ with $i_\rho(X)$ via the isomorphism $\psi'_{i_\rho(X)}$. In view of Lemma~\ref{lem:C}, the resulting triangle can be rotated to obtain the upper row. Similarly, for the lower row, we  just apply $i_\rho$ to (\ref{equ:rec2}) and use the isomorphism $\varepsilon'_{i_\lambda(X)}$.

For the commutativity of the rightmost square, we have
\begin{align*}
\xi j\circ i_\lambda \eta &=-\Sigma i_\rho j_\lambda \eta' j \circ i_\rho \sigma j_\rho j \circ \varepsilon' i_\lambda j_\rho j \circ i_\lambda \eta\\
&= - \Sigma i_\rho j_\lambda \eta' j\circ \Sigma i_\rho j_\lambda j \eta \circ i_\rho \sigma \circ \varepsilon' i_\lambda\\
&=- i_\rho \sigma \circ \varepsilon' i_\lambda.
\end{align*}
Here, the first equality uses the last statement in Proposition \ref{prop:N-xi}, the second one use the naturalness of $i_\rho \sigma \circ \varepsilon' i_\lambda\colon i_\lambda \rightarrow \Sigma i_\rho j_\lambda j$, and the last one uses the fact that $\eta' j\circ j\eta ={\rm Id}_{j}$. Similarly, one proves the commutativity of the leftmost square.
\end{proof}

The following observation seems to be new, and reveals a certain compatibility of the morphisms in (\ref{equ:rec1}) and (\ref{equ:rec2}).

\begin{cor}\label{cor:rec}
We have $\delta j_\lambda j \circ j_\rho \phi' j \circ \eta=-ii_\rho \sigma \circ i \varepsilon' i_\lambda \circ \psi.$
\end{cor}

\begin{proof}
By the rightmost square in (\ref{equ:N-xi}), we have the first  equality of the following identity.
\begin{align*}
\delta j_\lambda j \circ j_\rho \phi' j \circ \eta &= i\xi j \circ \psi j_\rho j \circ \eta \\
                                               &= i\xi j \circ ii_\lambda \eta \circ \psi\\
                                               &=- ii_\rho \sigma \circ i\varepsilon' i_\lambda \circ \psi
\end{align*}
Here, the second equality uses the naturalness of $\psi$, and the last one uses the rightmost square in (\ref{equ:C-xi}).
\end{proof}

\subsection{Mayer-Vietoris triangles}

Let $\mathcal{T}$ be a triangulated category with translation functor $\Sigma$. We denote by $\Sigma^{-1}$ the quasi-inverse of $\Sigma$.

\begin{lem}\label{lem:tr4}
Suppose that there are exact triangles in $\mathcal{T}$: $U\stackrel{f}\rightarrow V  \stackrel{b}\rightarrow Z \stackrel{z} \rightarrow \Sigma(U)$, $Y\stackrel{a}\rightarrow V \stackrel{g} \rightarrow W  \stackrel{w} \rightarrow \Sigma(Y)$, and $X\stackrel{x}\rightarrow U \xrightarrow{g\circ f}  W  \stackrel{y} \rightarrow \Sigma(X)$. Assume further that ${\rm Hom}_\mathcal{T}(X, \Sigma^{-1}(W))=0={\rm Hom}_\mathcal{T}(U, X)$. Then there are unique morphisms $u\colon X \rightarrow Y$ and $v\colon Z\rightarrow \Sigma(X)$ such that the following diagram commutes  and that the leftmost column is also exact.
\[
\xymatrix{
X\ar[r]^-{x} \ar@{.>}[d]_{u} & U\ar[d]^-{f} \ar[r]^-{g\circ f} & W \ar@{=}[d] \ar[r]^-{y} &\Sigma(X)\ar@{.>}[d]^-{\Sigma(u)}\\
Y\ar[d]_{b\circ a} \ar[r]^-{a} & V \ar[d]^-{b} \ar[r]^-{g} & W\ar[d]^-{y} \ar[r]^-{w}   & \Sigma(Y)\\
Z \ar@{.>}[d]_{v} \ar@{=}[r] & Z\ar[d]^{z} \ar@{.>}[r]^{-v} & \Sigma(X)\\
\Sigma(X) \ar[r]^-{\Sigma(x)} & \Sigma(U)
}\]
Moreover, both the triangles
$$X \xrightarrow{\binom{u}{-x}}  Y\oplus U \xrightarrow{(a, f)} V \xrightarrow {y \circ g}\Sigma(X) $$
and
$$V \xrightarrow{\binom{b}{-g}}  Z\oplus W \xrightarrow{(-v, y)} \Sigma(X) \xrightarrow {-\Sigma(f\circ x)}\Sigma(V) $$
are exact.
\end{lem}

\begin{proof}
The existence of $u$ and $v$ follows from the octahedral axiom (TR4). For the uniqueness of $u$, we assume that there is another morphism $u'\colon X\rightarrow Y$ such that  the square in the left upper corner commutes. Then $a\circ (u-u')=0$, which implies that $u-u'$ factors through $-\Sigma^{-1}(w)\colon \Sigma^{-1}(W)\rightarrow Y$. In view of ${\rm Hom}_\mathcal{T}(X, \Sigma^{-1}(W))=0$, we infer that $u=u'$. Similarly, we obtain the uniqueness of $v$. By the uniqueness of these two morphisms and \cite[Proposition 1.4.6]{Nee}, we infer the above two exact triangles.
\end{proof}

We assume that we are given the recollement (\ref{rec1}). Keep the assumptions and notation therein. We mention that the following Mayer-Vietoris triangles are partially obtained in \cite[Proposition 5.10]{Rou} with a different argument.

\begin{prop}\label{prop:MV}
 We have the following functorial exact triangles:
$$ii_\rho j_\lambda j \xrightarrow{\binom{ii_\rho \phi}{-\varepsilon j_\lambda j}} ii_\rho \oplus j_\lambda j \xrightarrow{(\varepsilon, \phi)} {\rm Id}_\mathcal{T} \xrightarrow{ \delta j_\lambda j\circ j_\rho \phi' j \circ \eta} \Sigma ii_\rho j_\lambda j$$
and
$${\rm Id}_\mathcal{T} \xrightarrow{\binom{\psi}{-\eta}} ii_\lambda \oplus j_\rho j \xrightarrow{(-ii_\rho \sigma \circ i\varepsilon' i_\lambda,\;  \delta j_\lambda j \circ j_\rho \phi' j)} \Sigma ii_\rho j_\lambda j \xrightarrow{-\Sigma(\phi\circ \varepsilon j_\lambda j)} \Sigma.$$
\end{prop}

\begin{proof}
Since $ji\simeq 0$, we obtain $${\rm Hom}_\mathcal{T}(ii_\rho j_\lambda j(X), \Sigma^{-1}j_\rho j(X))=0={\rm Hom}_\mathcal{T}(j_\lambda j(X), ii_\rho j_\lambda j(X))$$
 by the adjoint pairs $(j, j_\rho)$ and $(j_\lambda, j)$. Recall that $Nj=\eta\circ \phi$. The lower row of (\ref{equ:N-xi}) yields the following functorial exact triangle
 $$ii_\rho j_\lambda j \stackrel{\varepsilon j_\lambda j}\longrightarrow j_\lambda j \stackrel{\eta\circ \phi}\longrightarrow j_\rho j \xrightarrow{\delta j_\lambda j \circ j_\rho \phi' j} \Sigma ii_\rho j_\lambda j.$$
 We apply Lemma \ref{lem:tr4} to the above triangle together with  (\ref{equ:rec2}) and (\ref{equ:rec1}); compare \cite[1.4.7]{BBD}.
\[
\xymatrix{
ii_\rho j_\lambda j \ar[r]^-{\varepsilon j_\lambda j} \ar@{.>}[d]_{u} & j_\lambda j \ar[d]^-{\phi} \ar[r]^-{\eta \circ \phi} & j_\rho j  \ar@{=}[d] \ar[rr]^-{\delta j_\lambda j \circ j_\rho \phi' j} &&\Sigma i i_\lambda j_\rho j \ar@{.>}[d]^-{\Sigma(u)}\\
ii_\rho \ar[d]_{\psi \circ \varepsilon} \ar[r]^-{\varepsilon} & {\rm Id}_\mathcal{T} \ar[d]^-{\psi} \ar[r]^-{\eta} & j_\rho j \ar[d]^-{\delta j_\lambda j\circ j_\rho \phi' j} \ar[rr]^-{\delta}  & & \Sigma ii_\rho\\
ii_\lambda  \ar@{.>}[d]_{v} \ar@{=}[r] & ii_\lambda \ar[d]^{\sigma } \ar@{.>}[r]^-{-v} & \Sigma ii_\rho j_\lambda j\\
\Sigma ii_\rho j_\lambda j  \ar[r]^-{\Sigma \varepsilon j_\lambda j } & \Sigma j_\lambda j
}\]
The uniquely determined morphism $u$ has to be $ii_\rho \phi$. In view of Corollary \ref{cor:rec} and the central square, we infer that the uniquely determined  morphism $v$ has to be $ii_\rho \sigma \circ i \varepsilon' i_\lambda$. Then the required exact triangles follow immediately.
\end{proof}

\begin{rem}
(1) Recall that $iC=\psi\circ \varepsilon $. We observe that the leftmost column in the octahedral diagram might be obtained by applying $i$ to the lower row of (\ref{equ:C-xi}).\\
(2) The morphism $v$ that makes the bottom square in the octahedral diagram commute, is already unique. In other words,  $v$ has to be $ii_\rho \sigma \circ i \varepsilon' i_\lambda$. Then the commutativity of the central square yields another proof of Corollary \ref{cor:rec}.
\end{rem}

\section{Recollements and comma categories}

In this section, we prove Theorem A (= Theorem \ref{thm:A}). We begin with the comma category and the cone functor.

Let $G\colon \mathcal{A}\rightarrow \mathcal{B}$ be an additive functor between additive categories. By the \emph{kernel ideal} of $G$, we mean the class of all morphisms in $\mathcal{A}$ that are annihilated by $G$. For an ideal $\mathcal{I}$ of $\mathcal{A}$, we denote by $\mathcal{A}/\mathcal{I}$ the \emph{factor category}. Then the kernel ideal of the projection functor ${\rm Pr}\colon \mathcal{A}\rightarrow \mathcal{A}/\mathcal{I}$ coincides with $\mathcal{I}$. For an additive full subcategory  $\mathcal{S}$ of $\mathcal{A}$, we denote by $[\mathcal{S}]$ the idempotent ideal consisting of those morphisms factoring through objects in $\mathcal{S}$. The factor category ${\mathcal{A}/[\mathcal{S}]}$ is usually denoted by ${\mathcal{A}/\mathcal{S}}$.

Let $\mathcal{T}$ be a triangulated category and  $F\colon \mathcal{A}\rightarrow \mathcal{T}$ be  an additive  functor. We have two  \emph{comma categories} $(F\downarrow \mathcal{T})$ and $(\mathcal{T}\downarrow F)$. The objects in $(F\downarrow \mathcal{T})$ are triples $(A, X; f)$, where $A\in \mathcal{A}$ and $X\in \mathcal{T}$  are objects and $f\colon F(A)\rightarrow X$ is a morphism in $\mathcal{T}$. The morphisms $(a,b)\colon (A, X; f)\rightarrow (A', X';f')$ are given by morphisms $a\colon A\rightarrow A'$ in $\mathcal{A}$ and $b\colon X\rightarrow X'$ in $\mathcal{T}$ satisfying $f'\circ F(a)=b\circ f$. The composition of morphisms is defined naturally.

 We say that a morphism $(a, b)\colon (A, X; f)\rightarrow (A', X';f')$ is \emph{right trivial},  provided that $a=0$ and $b$ factors through $f'$ in $\mathcal{T}$. These morphisms form an ideal ${\rm RT}$ for $(F\downarrow \mathcal{T})$. Then we have the factor category ${(F\downarrow \mathcal{T})/{{\rm RT}}}$. We claim that the ideal ${\rm RT}$ is square zero. Indeed, for two right trivial morphisms $(0, b)\colon (A, X; f)\rightarrow (A', X';f')$ and $(0, b')\colon(A', X'; f')\rightarrow (A'', X'';f'')$, we have that $b$ factors through $f'$ and  that $b'\circ f'=f''\circ F(0)=0$. Then we infer that $b'\circ b=0$.

Dually, the objects in $(\mathcal{T}\downarrow F)$ are triples $(Y, A; g)$ with $Y\in \mathcal{T}$, $A\in \mathcal{A}$ and $g\colon Y\rightarrow F(A)$ a morphism in $\mathcal{T}$. The morphisms and their composition are defined similarly. A morphism $(c, a)\colon (Y, A;g)\rightarrow (Y', A';g')$ is \emph{left trivial},  provided that $a=0$ and $c\colon Y\rightarrow Y'$ factors through $g$ in $\mathcal{T}$. Such morphisms form an ideal ${\rm LT}$ of $(\mathcal{T}\downarrow F)$, which is also square zero. We have the corresponding factor category ${(\mathcal{T}\downarrow F)/{{\rm LT}}}$.

Given an object $(Y, A;g)$ in $(\mathcal{T}\downarrow F)$, we form an exact triangle in $\mathcal{T}$
$$Y\stackrel{g}\longrightarrow F(A) \stackrel{f}\longrightarrow  X\longrightarrow \Sigma(Y).$$
The obtained object $(A, X; f)$ in $(F\downarrow \mathcal{T})$ will be denoted by ${\rm Cone}(Y, A;g)$.  For a morphism $(c,a)\colon (Y, A;g)\rightarrow (Y', A';g')$, by the axiom (TR3) there is a commutative diagram in $\mathcal{T}$.
\begin{align}\label{equ:tr3}
\xymatrix{
Y \ar[r]^-{g} \ar[d]_{c} & F(A) \ar[d]^-{F(a)} \ar[r]^-{f} & X \ar@{.>}[d]^-{b} \ar[r] & \Sigma(Y)\ar[d]^-{\Sigma(c)}\\
Y \ar[r]^-{g'} & F(A') \ar[r]^-{f'} & X' \ar[r] & \Sigma(Y')
}
\end{align}
It is well known that the morphism $b\colon X \rightarrow X'$ is not unique in general. However, if there is another $b'$ making the diagram commute, their difference $b-b'$ necessarily factors through $f'$. Consequently, the following \emph{cone functor}
$${\rm Cone}\colon (\mathcal{T}\downarrow F)\longrightarrow {(F\downarrow \mathcal{T})/{\rm RT}}, \quad \quad (Y, A;g)\mapsto (A, X; f), \quad (c, a)\mapsto (a, b)$$
is well defined. Since the functor vanishes on the ideal ${\rm LT}$, it induces uniquely the following functor
$${\rm Cone}\colon {(\mathcal{T}\downarrow F)/{\rm LT}}\longrightarrow {(F\downarrow \mathcal{T})/{\rm RT}}.$$
Here, we abuse the notation.

\begin{lem}\label{lem:cone}
The cone functor ${\rm Cone}\colon {(\mathcal{T}\downarrow F)/{\rm LT}}\rightarrow {(F\downarrow \mathcal{T})/{\rm RT}}$ is an equivalence of categories.
\end{lem}

\begin{proof}
It is clear that the cone functor is dense. The axiom (TR3) implies that it  is also full. For the faithfulness, we just observe that in (\ref{equ:tr3}), the morphism $(c, a)$ lies in ${\rm LT}$ if and only if $(a, b)$ lies in ${\rm RT}$.
\end{proof}

In what follows, we assume that we are given the recollement (\ref{rec1}).  We keep the notation therein. In particular, we have the functor
$$i_\rho j_\lambda \colon \mathcal{T}''\longrightarrow \mathcal{T}'.$$
For each $X\in \mathcal{T}$, we have a morphism $i_\rho(\phi_X)\colon i_\rho j_\lambda j(X)\rightarrow i_\rho(X)$. It might be viewed as an object $(j(X), i_\rho(X); i_\rho(\phi_X))$ in the comma category $(i_\rho j_\lambda\downarrow \mathcal{T}')$. Therefore, we have the following well-defined functor
$$\Phi\colon \mathcal{T}\longrightarrow (i_\rho j_\lambda\downarrow \mathcal{T}'), \quad \quad X\mapsto (j(X), i_\rho(X); i_\rho(\phi_X)), \quad f\mapsto (j(f), i_\rho(f)).$$

The first main result compares the middle category $\mathcal{T}$ in the recollement (\ref{rec1}) with the comma category $(i_\rho j_\lambda\downarrow \mathcal{T}')$.

\begin{thm}\label{thm:A}
Keep the assumptions and notation as above. Then the functor
$$\Phi\colon \mathcal{T}\longrightarrow (i_\rho j_\lambda\downarrow \mathcal{T}')$$
is full and dense, and the kernel ideal is square zero. In particular, the functor $\Phi$ is an epivalence.
\end{thm}

\begin{proof}
For the fullness, we take an arbitrary  morphism
$$(a, b)\colon (j(X), i_\rho(X); i_\rho(\phi_X))\longrightarrow (j(Y), i_\rho(Y); i_\rho(\phi_Y))$$ in $(i_\rho j_\lambda \downarrow \mathcal{T}').$
 We have the following identities
\begin{align*}
\phi_Y \circ j_\lambda(a) \circ \varepsilon_{j_\lambda j(X)} &= \varepsilon_Y \circ ii_\rho(\phi_Y)\circ ii_\rho j_\lambda(a)\\
 &=\varepsilon_Y\circ i(b) \circ ii_\rho(\phi_X),
\end{align*}
where the first equality uses the naturalness of $\varepsilon$ and the second one uses the condition $i_\rho(\phi_Y)\circ i_\rho j_\lambda(a)=b\circ  i_\rho(\phi_X)$ of the given morphism $(a, b)$. In other words, we have
$$(\varepsilon_Y \circ i(b), \phi_Y \circ j_\lambda(a)) \circ \binom{ii_\rho(\phi_X)}{-\varepsilon_{j_\lambda j(X)}}=0.$$
Recall from  Proposition \ref{prop:MV} that we have the following exact triangle
\begin{align}\label{equ:triX}
ii_\rho j_\lambda j(X) \xrightarrow{\binom{ii_\rho(\phi_X)}{-\varepsilon_{j_\lambda j(X)}}} ii_\rho(X) \oplus j_\lambda j(X) \xrightarrow{(\varepsilon_X, \phi_X)} X \xrightarrow{ \delta_{j_\lambda j(X)}\circ j_\rho (\phi'_{j(X)}) \circ \eta_X} \Sigma ii_\rho j_\lambda j(X).
\end{align}
Hence, there is a morphism $f\colon X\rightarrow Y$ in $\mathcal{T}$ satisfying
$$(\varepsilon_Y \circ i(b), \phi_Y \circ j_\lambda(a))=f\circ (\varepsilon_X, \phi_X).$$
Then it follows that $i_\rho(f)=b$ and $j(f)=a$, proving the required fullness.

For the denseness, we take an object $(A, Z; f)$ in $(i_\rho j_\lambda \downarrow \mathcal{T}')$,  where $f\colon i_\rho j_\lambda(A) \rightarrow Z$ is a morphism in $\mathcal{T}'$. Consider the following exact triangle
$$ii_\rho j_\lambda (A) \xrightarrow{\binom{i(f)}{-\varepsilon_{j_\lambda(A)}}}
 i(Z)\oplus j_\lambda(A) \xrightarrow{(x, y)} C\longrightarrow \Sigma ii_\rho j_\lambda(A)$$
in $\mathcal{T}$. Applying $j$ to it, we infer that $j(y)$ is an isomorphism. Applying $i_\rho$ to it and using the fact that $i_\rho \varepsilon$ is an isomorphism, we infer that $i_\rho(x)$ is also an isomorphism.

Recall that $\phi'\colon {\rm Id}_{\mathcal{T}''} \rightarrow jj_\lambda$ and  $\varepsilon'\colon {\rm Id}_\mathcal{T'}\rightarrow i_\rho i$ are the units of $(j_\lambda, j)$ and $(i, i_\rho)$, respectively. They are both isomorphisms. We claim that
$$(j(y)\circ \phi'_A, i_\rho(x)\circ \varepsilon'_Z)\colon (A, Z; f)\longrightarrow (j(C), i_\rho(C); i_\rho(\phi_C))$$
is an isomorphism. It suffices to observe the following identity.
\begin{align*}
(i_\rho (x)\circ \varepsilon'_Z)\circ f   &= i_\rho(x) \circ i_\rho i  (f) \circ \varepsilon'_{i_\rho j_\lambda(A)}\\
&=i_\rho (y\circ \varepsilon_{j_\lambda(A)}) \circ \varepsilon'_{i_\rho j_\lambda (A)}\\
&= i_\rho(y)\\
&=i_\rho(y) \circ i_\rho \phi_{j_\lambda(A)}\circ i_\rho j_\lambda (\phi'_A)\\
&= i_\rho (\phi_C) \circ i_\rho j_\lambda j(y) \circ i_\rho j_\lambda(\phi'_A)\\
&= i_\rho (\phi_C)\circ i_\rho j_\lambda(j(y)\circ \phi'_A)
\end{align*}
Here, the first equality uses the naturalness of $\varepsilon'$, the second one uses the fact that $x\circ i(f)=y\circ \varepsilon_{j_\lambda(A)}$, and the third one uses the fact  that $i_\rho \varepsilon \circ \varepsilon' i_\rho$ is the identity transformation. Similarly, the fourth one uses the fact that $\phi j_\lambda \circ j_\lambda \phi'$ is the identity transformation. The fifth equality uses the naturalness of $\phi$.

Take two morphisms $f\colon X\rightarrow Y$ and $f'\colon Y\rightarrow W$ from the kernel ideal of $\Phi$. Since $j(f)=0$, it follows that $f$ factors through some object $i(A)$ for some $A\in \mathcal{T}'$. By $i_\rho(f')=0$, we infer that $f'$ factors through $j_\rho(B)$ for some $B\in \mathcal{T}''$. But we have ${\rm Hom}_\mathcal{T}(i(A), j_\rho(B))\simeq {\rm Hom}_{\mathcal{T}''}(ji(A), B)=0$. It follows that $f'\circ f=0$. This completes the proof.
\end{proof}

\begin{rem}
We apply $\mathcal{T}(-, Y)$ to the exact triangle (\ref{equ:triX}). Then the argument in the first paragraph of the proof actually yields the following exact sequence
$$\mathcal{T}(\Sigma ii_\rho j_\lambda j(X), Y) \longrightarrow \mathcal{T}(X, Y) \stackrel{\Phi}\longrightarrow (i_\rho j_\lambda\downarrow \mathcal{T}')(\Phi(X), \Phi(Y))\longrightarrow 0.$$
\end{rem}

Dually, we consider the functor
$$i_\lambda j_\rho\colon \mathcal{T}''\longrightarrow \mathcal{T}',$$
and then the comma category $(\mathcal{T}'\downarrow i_\lambda j_\rho)$. We have the following functor
$$\Psi\colon \mathcal{T}\longrightarrow (\mathcal{T}'\downarrow i_\lambda j_\rho),\quad \quad X\mapsto (i_\lambda(X), j(X); i_\lambda(\eta_X)), \quad f\mapsto (i_\lambda(f), j(f)).$$
It turns out that $\Psi$ is full and dense with a square-zero kernel ideal.

The functors $\Phi$ and $\Psi$ are related in the following diagram, which commutes up to natural isomorphisms.
\[\xymatrix{
\mathcal{T} \ar[d]_{\Psi} \ar[r]^-{\Phi} & (i_\rho j_\lambda \downarrow \mathcal{T}') \ar[r]^-{\rm Pr} & {(i_\rho j_\lambda \downarrow \mathcal{T}')/{\rm RT}}\ar[d]^-{\tilde{\Sigma}}\\
(\mathcal{T}' \downarrow i_\lambda j_\rho) \ar[r]^-{\rm Pr} & {(\mathcal{T}' \downarrow i_\lambda j_\rho)/{\rm LT}}
\ar[r]^-{\rm Cone} & {(i_\lambda j_\rho\downarrow \mathcal{T}')/{\rm RT}}
}\]
Here, the two ``${\rm Pr}$" denote the projection functors, and ``${\rm Cone}$" is the cone functor in Lemma \ref{lem:cone}. For the equivalence $\tilde{\Sigma}$, we recall the intertwining isomorphism $\xi\colon i_\lambda j_\rho \rightarrow \Sigma i_\rho j_\lambda$ in Proposition \ref{prop:N-xi}. The equivalence $\tilde{\Sigma}$ sends $(A, Z;  f)$ to $( A, \Sigma(Z); \Sigma(f)\circ \xi_A)$.

To verify the commutativity of the diagram, we rotate the upper row of (\ref{equ:C-xi}) to obtain the following exact triangle
$$i_\rho (X)\stackrel{C_X}\longrightarrow i_\lambda(X)\xrightarrow{i_\lambda(\eta_X)} i_\lambda j_\rho j(X) \xrightarrow{\Sigma (\psi'_{i_\rho(X)}) \circ i_\lambda(\delta_X)} \Sigma i_\rho(X).$$
So, we have
$${\rm Cone} \circ {\rm Pr}\circ \Psi (X)=(j(X), \Sigma i_\rho(X);  \Sigma(\psi'_{i_\rho(X)})\circ i_\lambda(\delta_X)). $$
On the other hand, we have
$$\tilde{\Sigma} \circ {\rm Pr}\circ \Phi(X)=(j(X), \Sigma i_\rho(X); \Sigma i_\rho(\phi_X)\circ \xi_{j(X)}).$$
By the leftmost square of (\ref{equ:C-xi}), we do have
$$ \Sigma(\psi'_{i_\rho(X)})\circ i_\lambda(\delta_X) = \Sigma i_\rho(\phi_X)\circ \xi_{j(X)}. $$
This proves the required commutativity.

\section{Morphic enhancements and module categories}

In this section, we apply Theorem \ref{thm:A} to morphic enhancements. Recall that the notion of a morphic enhancement is formally introduced in \cite[Appendix C]{Kr}. In fact, its study goes back to \cite[Section 6]{Ke91}, since a morphic enhancement appears as the first level in a tower of triangulated categories. We will prove Theorem B (= Theorem \ref{thm:B}), which relates the enhancing category to the module category over the given triangulated category.

 Let $\mathcal{T}$ and $\mathcal{T}_1$ be triangulated categories. By a \emph{morphic enhancement} of $\mathcal{T}$, we mean a fully faithful triangle functor $i\colon \mathcal{T}\rightarrow \mathcal{T}_1$ which fits into a recollement
\begin{align}\label{rec2}
\xymatrix{
\mathcal{T} \ar[rr]|{i} &&  \mathcal{T}_1 \ar[rr]|{j}  \ar@/_1pc/[ll]|{i_\lambda} \ar@/^1pc/[ll]|{i_\rho} &&  \mathcal{T}   \ar@/_1pc/[ll]|{j_\lambda} \ar@/^1pc/[ll]|{j_\rho}
}
\end{align}
such that $(i_\rho, j_\lambda)$ is an adjoint pair.

We keep the notation in the previous section:  $\varepsilon$ and $\phi$ denote the counits of the adjoint pairs $(i, i_\rho)$ and $(j_\lambda, j)$, respectively, and $\eta$ and $\psi$ denote the unit of $(j, j_\rho)$ and $(i_\lambda, i)$, respectively. In particular, $\xi\colon i_\lambda j_\rho \rightarrow \Sigma i_\rho j_\lambda$ denotes the intertwining isomorphism.

For the new adjoint pair $(i_\rho, j_\lambda)$, we will use $\theta\colon {\rm Id}_{\mathcal{T}_1}\rightarrow j_\lambda i_\rho$ and $\theta'\colon i_\rho j_\lambda\rightarrow {\rm Id}_\mathcal{T}$ to denote its unit and  counit, respectively. Then $\theta'$ is an isomorphism.

\begin{lem}\label{lem:Theta}
We have $(\Sigma j_\lambda \varepsilon' )^{-1}\circ \theta \Sigma i=  \Sigma \varepsilon j_\lambda \circ (\Sigma i \theta')^{-1}\colon \Sigma i\rightarrow \Sigma j_\lambda$.
\end{lem}

\begin{proof}
 Recall that $i_\rho\colon \mathcal{T}_1\rightarrow \mathcal{T}$ is identified with the Verdier quotient functor. So, it suffices to prove the following identity
 $$(\Sigma j_\lambda \varepsilon' i_\rho)^{-1}\circ \theta \Sigma i i_\rho=\Sigma \varepsilon j_\lambda i_\rho \circ (\Sigma i \theta' i_\rho)^{-1}.$$
 We use the identities $(\varepsilon' i_\rho)^{-1}=i_\rho \varepsilon$ and $(\theta' i_\rho)^{-1}=i_\rho \theta$ from the corresponding adjunctions, and identify $\theta \Sigma$ as $\Sigma \theta$. Then the above identity is equivalent to
 $$ \Sigma j_\lambda i_\rho \varepsilon \circ \Sigma \theta i i_\rho =\Sigma \varepsilon j_\lambda i_\rho \circ \Sigma ii_\rho \theta.$$
 This identity is standard, since both sides equal $\Sigma (\theta \circ \varepsilon)$.
\end{proof}

The following observation is essentially  due to \cite[Proposition C.3 a)]{Kr}.

\begin{lem}\label{lem:zeta}
Keep the notation as above. Then $(j_\rho, \Sigma^{-1}i_\lambda)$ is an adjoint pair. Moreover, its unit $\zeta'\colon {\rm Id}_\mathcal{T}\rightarrow (\Sigma^{-1}i_\lambda) j_\rho$ can be chosen such that $(\zeta')^{-1}=\theta'\circ \Sigma^{-1}\xi$ holds.
\end{lem}

We will always choose such a unit $\zeta'$ as above. The corresponding counit is denoted by  $\zeta\colon j_\rho (\Sigma^{-1} i_\lambda) \rightarrow {\rm Id}_{\mathcal{T}_1}$.

\begin{proof}
Since $i_\rho$ has a right adjoint $j_\lambda$, it follows that $j_\rho$ has a right adjoint, say $p$. Denote  the unit  of $(j_\rho, p)$ by $x\colon {\rm Id}_\mathcal{T}\rightarrow pj_\rho$, which is an isomorphism. We recall the isomorphism $\theta'\circ \Sigma^{-1}\xi\colon (\Sigma^{-1}i_\lambda) j_\rho \rightarrow {\rm Id}_\mathcal{T}$.  We observe that ${\rm Ker}\; p={\rm Im}\; j_\lambda={\rm Ker}\; i_\lambda$. Then both $p$ and $\Sigma^{-1}i_\lambda$ factor uniquely through the Verdier quotient functor ${\rm can}\colon \mathcal{T}_1\rightarrow \mathcal{T}_1/{{\rm Ker}\; p}$; moreover, the composition ${\rm can}\circ j_\rho$ is an equivalence. Combining these facts, we obtain a unique natural isomorphism $a\colon \Sigma^{-1}i_\lambda\rightarrow p$ such that
$$aj_\rho=x\circ \theta'\circ \Sigma^{-1}\xi.$$
Then the unit $\zeta'$ for $(j_\rho, \Sigma^{-1}i_\lambda)$ is given by $\zeta'=(aj_\rho)^{-1}\circ x$, which equals $(\theta'\circ \Sigma^{-1}\xi)^{-1}$.
\end{proof}

We have the following recollement.
\begin{align}\label{rec3}
\xymatrix{
\mathcal{T} \ar[rr]|{j_\lambda} &&  \mathcal{T}_1 \ar[rr]|{\Sigma^{-1}i_\lambda}  \ar@/_1pc/[ll]|{i_\rho} \ar@/^1pc/[ll]|{j} &&  \mathcal{T}   \ar@/_1pc/[ll]|{j_\rho} \ar@/^1pc/[ll]|{\Sigma i}
}\end{align}
It implies that $j_\lambda\colon \mathcal{T}\rightarrow \mathcal{T}_1$ is also a morphic enhancement. The corresponding functorial exact triangles are given by (\ref{equ:rec2}) and the following ones
\begin{align}\label{equ:rec3}
j_\rho \Sigma^{-1} i_\lambda(X) \xrightarrow{\zeta_X} X\xrightarrow{\theta_X} j_\lambda i_\rho(X) \stackrel{\partial_X}\longrightarrow j_\rho i_\lambda (X)
\end{align}
for each $X\in \mathcal{T}_1$.

By iterating the above argument, we infer that the functor $j_\rho\colon \mathcal{T}\rightarrow \mathcal{T}_1$ is  a  morphic enhancement,  which fits into the following  recollement.
\begin{align}\label{rec4}\xymatrix{
\mathcal{T} \ar[rr]|{j_\rho} &&  \mathcal{T}_1 \ar[rr]|{\Sigma^{-1}i_\rho}  \ar@/_1pc/[ll]|{j} \ar@/^1pc/[ll]|{\Sigma^{-1}i_\lambda} &&  \mathcal{T}   \ar@/_1pc/[ll]|{\Sigma i} \ar@/^1pc/[ll]|{\Sigma j_\lambda}
}\end{align}
The corresponding functorial exact triangles are given by (\ref{equ:rec3}) and (\ref{equ:rec1}).

We will consider the following functorial triangle
\begin{align}\label{equ:tri1}
j\stackrel{\alpha}\longrightarrow i_\rho \stackrel{\beta} \longrightarrow i_\lambda \stackrel{\gamma} \longrightarrow \Sigma j,
\end{align}
where $\alpha=i_\rho \phi \circ (\theta' j)^{-1}$, $\beta=i_\lambda \varepsilon \circ (\psi' i_\rho)^{-1}$ and $\gamma=\Sigma j \zeta \circ (\eta' i_\lambda)^{-1}$. In view of Lemma \ref{lem:C},  we observe that $\alpha$, $\beta$ and $\Sigma^{-1}\gamma$ are precisely the conorm morphisms of the recollements (\ref{rec3}), (\ref{rec2}) and (\ref{rec4}), respectively.

The third statement of the following lemma is due to \cite[Proposition C.3 b)]{Kr} in a slightly different setting.

\begin{prop}
Keep the notation and assumptions as above. Then the following statements hold.
\begin{enumerate}
\item The intertwining isomorphism for (\ref{rec3}) is given by $-\Sigma(\varepsilon'\circ \eta')^{-1}\colon i_\rho (\Sigma i)=\Sigma i_\rho i\rightarrow \Sigma jj_\rho$.
\item We have $\gamma=\Sigma \theta' j \circ i_\rho \sigma \circ \varepsilon' i_\lambda$.
\item The functorial triangle (\ref{equ:tri1}) is exact.
\end{enumerate}
\end{prop}

\begin{proof}
We denote by $\bar{\xi}$ the intertwining isomorphism of (\ref{rec3}). Applying Proposition \ref{prop:N-xi} and Remark \ref{rem:1}, $\bar{\xi}$ is uniquely determined by the following commutative diagram.
\[\xymatrix{
\Sigma i \ar@{=}[d]\ar[rr]^-{\theta \Sigma i} && j_\lambda i_\rho (\Sigma i)=\Sigma j_\lambda i_\rho i \ar@{.>}[d]^-{j_\lambda \bar{\xi}}\\
\Sigma i \ar[rr]^-{\sigma j_\rho \circ \Sigma i \zeta'} && \Sigma j_\lambda j j_\rho
}\]
Then (1) follows from the following identity
\begin{align*}
j_\lambda \Sigma  (\varepsilon'\circ \eta')^{-1} \circ \theta \Sigma i & =(\Sigma j_\lambda \eta')^{-1}\circ (\Sigma j_\lambda \varepsilon')^{-1} \circ \theta \Sigma i\\
                                                               &= (\Sigma j_\lambda \eta')^{-1} \circ \Sigma \varepsilon j_\lambda \circ (\Sigma i \theta')^{-1}\\
                                                               &=-\sigma j_\rho \circ (i \xi)^{-1}\circ  (\Sigma i \theta')^{-1}\\
                                                               &=-\sigma j_\rho \circ \Sigma i \zeta',
\end{align*}
where the second equality uses Lemma \ref{lem:Theta}, the third one uses the leftmost square in (\ref{equ:N-xi}) and the last one uses the choice of $\zeta'$ in Lemma \ref{lem:zeta}.

We apply Proposition \ref{prop:C-xi} to (\ref{rec3}) and obtain the following commutative diagram.
\[
\xymatrix{
\Sigma^{-1} i_\rho i i_\lambda \ar[d]_{-\Sigma^{-2}\bar{\xi} i_\lambda}\ar[rr]^-{-\theta' j\circ \Sigma^{-1}i_\rho \sigma} && j \ar@{=}[d]\ar[r]^-{\alpha} & i_\rho  \ar@{=}[d]  \ar[rr]^-{i_\rho \psi} && i_\rho (\Sigma i)(\Sigma^{-1} i_\lambda)=i_\rho i i_\lambda \ar[d]^-{-\Sigma^{-1}\bar{\xi} i_\lambda}\\
\Sigma^{-1} j j_\rho i_\lambda \ar[rr]^-{j\zeta} && j  \ar[r]^-{\alpha} & i_\rho \ar[rr]^-{j\partial \circ \phi' i_\rho} && \Sigma j j_\rho (\Sigma^{-1}i_\lambda)=jj_\rho i_\lambda
}\]
 Here, for the leftmost vertical arrow, we identify $\bar{\xi}(\Sigma^{-1} i_\lambda)$ with $\Sigma^{-1} \bar{\xi} i_\lambda$. Note that $-\Sigma^{-1}\bar{\xi}=(\varepsilon'\circ \eta')^{-1}$. Then (2) follows immediately from the leftmost square and the definition of $\gamma$.

By (2), we have the following commutative diagram.
 \[\xymatrix{
 i_\rho j_\lambda j \ar[d]_-{\theta' j} \ar[r]^-{i_\rho \phi} & i_\rho \ar@{=}[d] \ar[r]^-{\beta} & i_\lambda \ar@{=}[d]  \ar[rr]^-{i_\rho \sigma \circ \varepsilon' i_\lambda} && \Sigma i_\rho j_\lambda j \ar[d]^-{\Sigma \theta' j}\\
 j \ar[r]^-\alpha & i_\rho \ar[r]^-\beta & i_\lambda \ar[rr]^-\gamma && \Sigma j
 }\]
 The upper row is exact since it is the same as the lower row of (\ref{prop:C-xi}). Then (3) follows immediately.
\end{proof}

Recall that the comma category $({\rm Id}_\mathcal{T}\downarrow \mathcal{T}) $ is just the \emph{morphism category} ${\rm mor}(\mathcal{T})$ of $\mathcal{T}$. The objects of ${\rm mor}(\mathcal{T})$ are the morphisms in $\mathcal{T}$, and the morphisms are given by commutative squares in $\mathcal{T}$. We identify $i_\rho j_\lambda$ with ${\rm Id}_\mathcal{T}$ via the isomorphism  $\theta'\colon i_\rho j_\lambda \rightarrow {\rm Id}_\mathcal{T}$. Then the epivalence  in Theorem \ref{thm:A} for the recollement (\ref{rec2}) yields the following epivalence
$$\Phi\colon \mathcal{T}_1 \longrightarrow {\rm mor}(\mathcal{T}), \quad X\mapsto (j(X), i_\rho(X); i_\rho(\phi_X)\circ (\theta'_{j(X)})^{-1})=(j(X), i_\rho(X); \alpha_X).$$
We mention that this epivalence is due to \cite[Theorem C.1]{Kr}.

Similarly, for the recollement (\ref{rec3}), we use the isomorphism $\eta'\colon jj_\rho \rightarrow {\rm Id}_\mathcal{T}$ and obtain the following epivalence
\begin{align*}
\Phi_\lambda\colon \mathcal{T}_1 \longrightarrow {\rm mor}(\mathcal{T}), \quad X\mapsto  &(\Sigma^{-1}i_\lambda(X), j(X); j(\zeta_X)\circ (\eta'_{\Sigma^{-1}i_\lambda(X)})^{-1})\\
& =(\Sigma^{-1}i_\lambda(X), j(X); \Sigma^{-1}(\gamma_X)).
\end{align*}
For the recollement (\ref{rec4}), we use the isomorphism $\psi'\colon (\Sigma^{-1} i_\lambda) (\Sigma i)= i_\lambda i\rightarrow {\rm Id}_\mathcal{T}$ and obtain the following epivalence
\begin{align*}
\Phi_\rho\colon \mathcal{T}_1 \longrightarrow {\rm mor}(\mathcal{T}), \quad X\mapsto  & (\Sigma^{-1}i_\rho(X), \Sigma^{-1} i_\lambda(X); \Sigma^{-1} i_\lambda(\varepsilon_X) \circ (\Sigma^{-1}\psi'_{i_\rho(X)})^{-1})\\
& = (\Sigma^{-1}i_\rho(X), \Sigma^{-1} i_\lambda(X); \Sigma^{-1}(\beta_X)).
\end{align*}

Denote by ${\rm mod}\mbox{-}\mathcal{T}$ the category of finitely presented contravariant functors from $\mathcal{T}$ to the category  of abelian groups, also known as the \emph{module category} over $\mathcal{T}$. It is a Frobenius abelian category. For each $A\in \mathcal{T}$, the corresponding representable functor is denoted by $\mathcal{T}(-, A)$. Denote by $\underline{\rm mod}\mbox{-} \mathcal{T}$ the stable category modulo these representable functors. Denote by $\Omega\colon \underline{\rm mod}\mbox{-} \mathcal{T}\rightarrow \underline{\rm mod}\mbox{-} \mathcal{T}$ the syzygy functor, and by $\Omega^{-1}$ its quasi-inverse.

It is well known that the following functor
$${\rm Cok}\colon {\rm mor}(\mathcal{T}) \longrightarrow {\rm mod}\mbox{-}\mathcal{T}, \quad (A, B; u)\mapsto {\rm Cok}(\mathcal{T}(-, A)\xrightarrow{\mathcal{T}(-, u)} \mathcal{T}(-, B)),$$
is full and dense; compare the proof of \cite[Theorem 1.1]{AR}. We denote by ${\rm Hpt}$ the ideal of ${\rm mor}(\mathcal{T})$ consisting of those morphisms $(a, b)\colon (A, B; f) \rightarrow (A', B'; f')$ such that there exists a morphism $h\colon B\rightarrow A'$ such that $f'\circ h\circ f=f'\circ a=b\circ f$.  In view of \cite[Proposition IV.1.6]{ARS}, the functor ``${\rm Cok}$" induces an equivalence
$${\rm Cok}\colon  {{\rm mor}(\mathcal{T})/{\rm Hpt}} \stackrel{\sim}\longrightarrow \underline{\rm mod}\mbox{-}\mathcal{T}.$$
We observe that ${\rm RT}, {\rm LT}\subseteq {\rm Hpt}$. Moreover, the cone functor as in Lemma \ref{lem:cone}
$${\rm Cone}\colon {{\rm mor}(\mathcal{T})/{\rm Hpt}}\longrightarrow {{\rm mor}(\mathcal{T})/{\rm Hpt}}$$
is also well-defined. The following commutative diagram is standard; compare \cite[Appendix B]{Kr05}.
\[\xymatrix{
{{\rm mor}(\mathcal{T})/{\rm Hpt}} \ar[d]_-{\rm Cone} \ar[rr]^-{\rm Cok} && \underline{\rm mod}\mbox{-} \mathcal{T} \ar[d]^-{\Omega^{-1}}\\
{{\rm mor}(\mathcal{T})/{\rm Hpt}} \ar[rr]^-{\rm Cok} && \underline{\rm mod}\mbox{-} \mathcal{T}
}\]

The following second main result relates the enhancing category $\mathcal{T}_1$ to the module category over $\mathcal{T}$.

\begin{thm}\label{thm:B}
Keep the above notation. Then the following statements hold.
\begin{enumerate}
\item The composite functor ${\rm Cok}\circ \Phi$ is full and dense, and induces equivalences
$$\mathcal{T}_1/{({\rm Im}\; j_\lambda + {\rm Im}\; j_\rho)} \stackrel{\sim}\longrightarrow {\rm mod}\mbox{-}\mathcal{T} \mbox{  and  } \mathcal{T}_1/{({\rm Im}\; i +{\rm Im}\; j_\lambda + {\rm Im}\; j_\rho)} \stackrel{\sim}\longrightarrow \underline{\rm mod}\mbox{-}\mathcal{T}.$$

\item The composite functor ${\rm Cok}\circ \Phi_\lambda$ is full and dense, and induces  equivalences
$$\mathcal{T}_1/{({\rm Im}\; j_\rho + {\rm Im}\; i)} \stackrel{\sim}\longrightarrow {\rm mod}\mbox{-}\mathcal{T} \mbox{  and  } \mathcal{T}_1/{({\rm Im}\; i +{\rm Im}\; j_\lambda + {\rm Im}\; j_\rho)} \stackrel{\sim}\longrightarrow \underline{\rm mod}\mbox{-}\mathcal{T}.$$

\item The composite functor ${\rm Cok}\circ \Phi_\rho$ is full and dense, and induces  equivalences
$$\mathcal{T}_1/{({\rm Im}\; i + {\rm Im}\; j_\lambda)} \stackrel{\sim}\longrightarrow {\rm mod}\mbox{-}\mathcal{T} \mbox{  and  } \mathcal{T}_1/{({\rm Im}\; i +{\rm Im}\; j_\lambda + {\rm Im}\; j_\rho)} \stackrel{\sim}\longrightarrow \underline{\rm mod}\mbox{-}\mathcal{T}.$$
\end{enumerate}
\end{thm}

\vskip 3pt

\begin{proof}
Observe that ${\rm Im}\; \Sigma i={\rm Im}\; i$ and ${\rm Im}\; \Sigma j_\lambda={\rm Im}\; j_\lambda$. Since all the functors $i$, $j_\lambda$ and $j_\rho$ are morphic enhancements of $\mathcal{T}$, it suffices to prove (1).

Recall that ``${\rm Cok}$" is full and dense. By Theorem \ref{thm:A}, we infer that the composite functor ${\rm Cok}\circ \Phi$ is full and dense. We claim that the kernel ideal of ${\rm Cok}\circ \Phi$ equals $[{\rm Im}\; j_{\lambda} + {\rm Im}\; j_\rho]$, the idempotent ideal given by ${\rm Im}\; j_{\lambda} + {\rm Im}\; j_\rho$.

Recall that $\phi_{j_\lambda(A)}$  is an isomorphism for each $A\in \mathcal{T}$. It follows that ${\rm Cok}\circ \Phi$ vanishes on $j_\lambda(A)$. The vanishing on ${\rm Im}\; j_\rho$ is clear by the fact $i_\rho j_\rho \simeq 0$. Therefore, the kernel ideal contains $[{\rm Im}\; j_{\lambda} + {\rm Im}\; j_\rho]$.

For the converse inclusion, we take a morphism $f\colon X\rightarrow Y$ in the kernel ideal. By the vanishing of ${\rm Cok}$ on $\Phi(f)$, we infer that $i_\rho(f)$ factors through $i_\rho(\phi_Y)\circ (\theta'_{j(Y)})^{-1}$, or equivalently, through $i_\rho(\phi_Y)$. By the adjoint pair $(i_\rho, j_\lambda)$, we infer that $\theta_{Y} \circ f=j_\lambda i_\rho(\phi_Y)\circ x$ for some morphism $x\colon X\rightarrow j_\lambda i_\rho j_\lambda j(Y)$.

Recall that $\alpha=i_\rho \phi \circ (\theta' j)^{-1} \colon j\rightarrow i_\rho$ is the conorm morphism for (\ref{rec3}); compare Lemma \ref{lem:C}.  Then we have $j_\lambda \alpha=\theta \circ \phi$. Consequently, we have
$$j_\lambda i_\rho (\phi_Y)=\theta_Y \circ \phi_Y \circ j_\lambda(\theta'_{j(Y)}).$$
Then we  infer that
$$\theta_Y\circ (f-\phi_Y\circ j_\lambda(\theta'_{j(Y)})\circ x)=0.$$
It follows by the exact triangle (\ref{equ:rec3}) that $f- \phi_Y\circ j_\lambda(\theta'_{j(Y)})\circ x$ factor through $\zeta_Y$. In particular, we infer that $f$ factors through the object $j_\lambda j(Y)\oplus j_\rho \Sigma^{-1} i_\lambda(Y)$. This implies that $f$ lies in $[{\rm Im}\; j_{\lambda} + {\rm Im}\; j_\rho]$. This proves the claim and the equivalence on the left.

For the equivalence on the right, we observe that
$$\Phi(iX)=(0, i_\rho i(X); 0)\simeq (0, X;0)$$
for each $X\in \mathcal{T}$. Therefore, we have ${\rm Cok}\circ \Phi (iX)\simeq {\rm Cok}(0, X; 0)$, which is further isomorphic to the representable functor $\mathcal{T}(-, X)$. Then the required equivalence follows immediately.
\end{proof}

\begin{rem}
The above equivalences are related as shown by the following commutative diagram.
\[\xymatrix{
\mathcal{T}_1 \ar@{=}[d] \ar[r]^-{\Phi \circ \Sigma^{-1}} & {\rm mor}(\mathcal{T}) \ar[r]^-{\rm Pr} & {{\rm mor}(\mathcal{T})/{\rm Hpt}} \ar[d]^-{\rm Cone}\ar[r]^-{\rm Cok} & \underline{\rm mod}\mbox{-}\mathcal{T} \ar[d]^-{\Omega^{-1}}\\
\mathcal{T}_1 \ar@{=}[d]\ar[r]^-{\Phi_\rho} & {\rm mor}(\mathcal{T}) \ar[r]^-{\rm Pr} & {{\rm mor}(\mathcal{T})/{\rm Hpt}} \ar[d]^-{\rm Cone} \ar[r]^-{\rm Cok} & \underline{\rm mod}\mbox{-}\mathcal{T} \ar[d]^-{\Omega^{-1}}\\
\mathcal{T}_1 \ar@{=}[d] \ar[r]^-{\Phi_\lambda} & {\rm mor}(\mathcal{T}) \ar[r]^-{\rm Pr} & {{\rm mor}(\mathcal{T})/{\rm Hpt}} \ar[d]^-{\rm Cone} \ar[r]^-{\rm Cok} & \underline{\rm mod}\mbox{-}\mathcal{T} \ar[d]^-{\Omega^{-1}}\\
\mathcal{T}_1 \ar[r]^-{\Phi} & {\rm mor}(\mathcal{T}) \ar[r]^-{\rm Pr} & {{\rm mor}(\mathcal{T})/{\rm Hpt}} \ar[r]^-{\rm Cok} & \underline{\rm mod}\mbox{-}\mathcal{T}
}\]
Here, the ``${\rm Pr}$'s " denote the projection functors. For the commutative squares on the left, we observe from the exact triangle (\ref{equ:tri1}) that ${\rm Cone}(\Sigma^{-1}\alpha)=\Sigma^{-1}\beta$, ${\rm Cone}(\Sigma^{-1}{\beta})=\Sigma^{-1}{\gamma}$ and ${\rm Cone}(\Sigma^{-1}\gamma)=\alpha$. In view of the commutative outer square, we recall the well-known fact that $\Omega^{-3}$ is naturally isomorphic to the functor $M\mapsto M\circ \Sigma^{-1}$  for each $M\in \underline{\rm mod}\mbox{-}\mathcal{T}$; for example, see \cite[Proposition B.2]{Kr05}.
\end{rem}

\section{Examples}

We will apply Theorem \ref{thm:B} to various inflation categories. We mention that inflation categories are studied in \cite{Ke90, Ke91} and \cite{RS06, LZ, IKM, Chen11} with quite different viewpoints. In Proposition \ref{prop:wpl}, we relate the weighted projective line of weight type $(2, 3 ,p)$ to the $\mathbb{Z}$-graded preprojective algebra of type $\mathbb{A}_{p-1}$.

\subsection{The inflation category}

We follow \cite[Appendix A]{Ke90} for the terminology on exact categories. Let $\mathcal{A}$ be a Frobenius exact category and $\underline{\mathcal{A}}$ its stable category modulo projective objects. We denote by ${\rm inf}(\mathcal{A})$ the full subcategory of the morphism category ${\rm mor}(\mathcal{A})$ consisting of inflations, called the \emph{inflation category}. In other words, an object $(X_1, X_0; f)$ belongs to ${\rm inf}(\mathcal{A})$ if and only if $f\colon X_1\rightarrow X_0$ is an inflation in $\mathcal{A}$.

We observe that ${\rm inf}(\mathcal{A})$ is a Frobenius category, whose conflations are precisely those sequences in ${\rm inf}(\mathcal{A})$
$$(X_1, X_0; f) \longrightarrow (Y_1, Y_0; g) \longrightarrow (Z_1, Z_0; h),$$
whose domain and codomain sequences are conflations in $\mathcal{A}$. An object $(P_1, P_0; f)$ is projective in ${\rm inf}(\mathcal{A})$ if and only if both $P_i$ are projective in $\mathcal{A}$; therefore, it is isomorphic to $(P_1, P_1; {\rm Id}_{P_1})\oplus (0, Q; 0)$, where $Q$ is the cokernel of the inflation $f\colon P_1\rightarrow P_0$; compare \cite[Lemma 2.1]{Chen11}.  We denote by $\underline{\rm inf}(\mathcal{A})$ its stable category. Denote by $p_i\colon \underline{\rm inf}(\mathcal{A})\rightarrow \underline{\mathcal{A}}$ the obvious functors such that $p_i(X_1, X_0; f)=X_i$ for $i=0, 1$.

The canonical embedding $i\colon \underline{\mathcal{A}}\rightarrow \underline{\rm inf}(\mathcal{A})$ sending $A$ to $(A, A; {\rm Id}_A)$ turns out be a morphic enhancement; compare  \cite[6.1]{Ke90} and \cite[C.2]{Kr}. Indeed, we have the following recollement.
\begin{align*}
\xymatrix{
\underline{\mathcal{A}} \ar[rr]|{i} &&   \underline{\rm inf}(\mathcal{A})  \ar[rr]|{j}  \ar@/_1pc/[ll]|{p_0} \ar@/^1pc/[ll]|{p_1} &&  \underline{\mathcal{A}}  \ar@/_1pc/[ll]|{j_\lambda} \ar@/^1pc/[ll]|{j_\rho}
}
\end{align*}
Here, the functors $j_\lambda$ and $j_\rho$ are given by $j_\lambda(A)=(A, Q(A); i_A)$ and $j_\rho(A)=(0, \Omega^{-1}(A); 0)$, where for each object $A$, $i_A\colon A\rightarrow Q(A)$ is a chosen inflation with $Q(A)$ projective. The functor $j$ is given by $j(X_1, X_0, f)=\Omega{\rm Cok}(f)$.

Applying Theorem \ref{thm:B}, we obtain three functors
$$\underline{\rm inf}(\mathcal{A})\longrightarrow {\rm mod}\mbox{-}\underline{\mathcal{A}},$$
which realize ${\rm mod}\mbox{-}\underline{\mathcal{A}}$ as certain factor categories of $\underline{\rm inf}(\mathcal{A})$. This recovers \cite[Corollary 1.3]{Lin} in slightly different terminologies. For example, the composite functor ${\rm Cok}\circ \Phi_\rho \circ \Sigma$ is given by
$$\underline{\rm inf}(\mathcal{A})\longrightarrow {\rm mod}\mbox{-}\underline{\mathcal{A}},\quad (X_1, X_0; f)\mapsto {\rm Cok}(\underline{\mathcal{A}}(-, f)\colon \underline{\mathcal{A}}(-, X_1) \rightarrow \underline{\mathcal{A}}(-, X_0)).$$

We mention that \cite[Corollary 1.3]{Lin} extends the main results in \cite{Eir}. Let $A$ be a selfinjective artin algebra and ${\rm mod}\mbox{-}A$ be the category of finitely generated right $A$-modules. The corresponding inflation subcategory ${\rm inf}({\rm mod}\mbox{-}A)$ is usually denoted by $\mathcal{S}(A)$, known as the \emph{submodule category} over $A$. We assume that $A$ is of finite representation type. Let $E$ be the direct sum of representatives of non-projective indecomposable $A$-modules. Then $\underline{\rm End}_A(E)=\Lambda$ is called the \emph{stable Auslander algebra} of $A$. There is a well-known equivalence
$$ {\rm mod}\mbox{-}(\underline{{\rm mod}}\mbox{-}A) \stackrel{\sim}\longrightarrow {\rm mod}\mbox{-}\Lambda,\quad F\mapsto F(E).$$
Therefore, we obtain three functors
$$\underline{\mathcal{S}}(A)\longrightarrow {\rm mod}\mbox{-}\Lambda, $$
each of which realizes ${\rm mod}\mbox{-}\Lambda$ as a factor category of $\underline{\mathcal{S}}(A)$ by an idempotent ideal. Two of these functors are studied in  \cite{Eir}.

\subsection{Invariant subspaces and weighted projective lines}

Let $p\geq 2$ and $A=k[t]/(t^p)$ be the truncated polynomial algebra over a field $k$. The category $\mathcal{S}(p)=\mathcal{S}(A)$ is extensively studied in \cite{RS06, RS08}, known as the category of \emph{invariant subspaces} of nilpotent operators with nilpotency index at most $p$. The  additive generator of $\underline{\rm mod}\mbox{-} A$ is given by $E=\oplus_{b=0}^{p-2} k[t]/(t^{b+1})$.  The corresponding stable Auslander algebra is isomorphic to the the \emph{preprojective algebra} $\Pi_{p-1}$ of type $\mathbb{A}_{p-1}$, which is given by the following quiver
\[\xymatrix{
1 \ar@/^1pc/[r]|{a_1} & 2 \ar@/^1pc/[l]|{b_1}  \ar@/^1pc/[r]|{a_2} & 3\ar@/^1pc/[l]|{b_2} \ar@/^1pc/[r] & \cdots  \ar@/^1pc/[l]\ar@/^1pc/[r]  & p-2 \ar@/^1pc/[r]|{a_{p-2}} \ar@/^1pc/[l] & p-1 \ar@/^1pc/[l]|{b_{p-2}}
}\]
subject to the relations $b_1 a_1=0=a_{p-2}b_{p-2}$ and $a_ib_i=b_{i+1}a_{i+1}$ for $1\leq i\leq p-3$. Then the above three functors
$$\underline{\mathcal{S}}(p)\longrightarrow {\rm mod}\mbox{-}\Pi_{p-1}. $$
are studied in \cite{RZ}; in particular, see \cite[Section 9]{RZ}.

For the graded version, we view $A=k[t]/(t^p)$ as a $\mathbb{Z}$-graded algebra with ${\rm deg}\; t=1$. Denote by ${\rm mod}^\mathbb{Z}\mbox{-}A$ the category of graded $A$-modules. The degree-shift functor $s\colon {\rm mod}^\mathbb{Z}\mbox{-}A \rightarrow {\rm mod}^\mathbb{Z}\mbox{-}A$ sends a graded $A$-module $M$ to another graded $A$-module $s(M)$, which is graded by $s(M)_n=M_{n+1}$ for $n\in \mathbb{Z}$. The corresponding inflation category ${\rm inf}({\rm mod}^\mathbb{Z}\mbox{-}A)$ is denoted by $\mathcal{S}^\mathbb{Z}(p)$. The degree-shift functor $s$ is defined naturally on $\mathcal{S}^\mathbb{Z}(p)$ and its stable category $\underline{\mathcal{S}}^\mathbb{Z}(p)$. We denote by $\tau$ the Auslander-Reiten translation on  $\underline{\mathcal{S}}^\mathbb{Z}(p)$.

The preprojective algebra $\Pi_{p-1}$ is $\mathbb{Z}$-graded such that ${\rm deg}\; a_i=1$ and ${\rm deg}\; b_i=0$. Therefore, we have the above three functors
$$\underline{\mathcal{S}}^\mathbb{Z}(p)\longrightarrow {\rm mod}^\mathbb{Z}\mbox{-}\Pi_{p-1}. $$
They all induce equivalences
\begin{align}\label{equ:S-pi}
\underline{\mathcal{S}}^\mathbb{Z}(p)/\mathcal{U}^\mathbb{Z}\stackrel{\sim}\longrightarrow \underline{\rm mod}^\mathbb{Z}\mbox{-}\Pi_{p-1},
\end{align}
where $\mathcal{U}^\mathbb{Z}$ is the smallest additive subcategory of $\underline{\mathcal{S}}^\mathbb{Z}(p)$  containing those objects $(X, X; {\rm Id}_X)$, $(0, X; 0)$ and $(X, Q(X); i_X)$, where $X\in {\rm mod}^\mathbb{Z}\mbox{-}A$ and $i_X\colon X\rightarrow Q(X)$ is the injective hull of $X$.

We denote by $\mathbb{X}(2, 3, p)$ the \emph{weighted projective line} of weight type $(2, 3, p)$; see \cite{GL}. The Picard group is a rank one abelian group
$$L(2, 3, p)=\langle \vec{x}_1, \vec{x}_2, \vec{x}_3\; |\; 2\vec{x}_1=3\vec{x}_2=p\vec{x}_3\rangle.$$
The line bundles are given by $\{\mathcal{O}(\vec{l})\; |\; \vec{l}\in L(2, 3, p)\}$, where $\mathcal{O}=\mathcal{O}(\vec{0})$ is the structure sheaf. The dualizing element $\vec{\omega}=\vec{x}_1-\vec{x}_2-\vec{x}_3$ plays a central role, since the Picard shift functor $\mathcal{F}\mapsto \mathcal{F}(\vec{\omega})$ yields the Auslander-Reiten translation on the category of coherent sheaves on $\mathbb{X}(2,3, p)$.

The category ${\rm vect}\mbox{-}\mathbb{X}(2, 3, p)$ of vector bundles is naturally a Frobenius exact category, where the projective-injective objects are precisely direct sums of line bundles; see \cite{KLM1}. Denote by $\underline{\rm vect}\mbox{-}\mathbb{X}(2,3,p)$ the stable category of vector bundles.

We denote by $\mathcal{V}_2$ the smallest additive full subcategory of  $\underline{\rm vect}\mbox{-}\mathbb{X}(2,3,p)$ containing all vector bundles of rank two. For each pair $(a, b)$ satisfying $a=0, 1$ and $0\leq b\leq p-2$, there is a unique non-split extension
$$0\longrightarrow \mathcal{O}(\vec{\omega}) \longrightarrow \mathcal{E}\langle a\vec{x}_2+b\vec{x}_3\rangle \longrightarrow \mathcal{O}(a\vec{x}_2+b\vec{x}_3)\longrightarrow 0.$$
These vector bundles $\mathcal{E}\langle a\vec{x}_2+b\vec{x}_3\rangle$ are called the \emph{extension bundles}; see \cite[Definition 4.1]{KLM1}.

The main result of \cite{KLM2}  states that there is an equivalence of triangulated categories
$$\Theta\colon \underline{\rm vect}\mbox{-}\mathbb{X}(2,3,p) \stackrel{\sim}\longrightarrow \underline{\mathcal{S}}^\mathbb{Z}(p).$$
The equivalence $\Theta$ is partially recovered in \cite[Examples 3.3 and 4.3]{Chen12}. Under this equivalence, the Picard shifts by $\vec{\omega}$ and by $\vec{x}_3$ correspond to the functors $\tau$ and $s$ on $\underline{\mathcal{S}}^\mathbb{Z}(p)$, respectively. More precisely, for each vector bundle $\mathcal{F}$, we have
\begin{align}\label{equ:theta}
\Theta(\mathcal{F}(\vec{\omega}))\simeq \tau \Theta(\mathcal{F}) \mbox{ and } \Theta(\mathcal{F}(\vec{x}_3))\simeq s\Theta(\mathcal{F}).
\end{align}

The first statement of the following observation is implicit in \cite[Lemma 5.7]{KLM2}.

\begin{prop}\label{prop:wpl}
We have $\Theta(\mathcal{V}_2)=\mathcal{U}^\mathbb{Z}$. Consequently, we have an equivalence
$$\underline{\rm vect}\mbox{-}\mathbb{X}(2, 3, p)/\mathcal{V}_2 \stackrel{\sim}\longrightarrow \underline{\rm mod}^\mathbb{Z}\mbox{-}\Pi_{p-1}.$$
\end{prop}

\begin{proof}
Recall from \cite[Lemma 5.2]{KLM2} that $L(2,3, p)$ is generated by $\vec{\omega}$ and $\vec{x}_3$. By \cite[Theorem 4.2]{KLM1},  each indecomposable vector bundle $\mathcal{F}$ of rank two is of the form $\mathcal{E}(\vec{l})$ for some extension bundle $\mathcal{E}$ and $\vec{l}\in L(2,3,p)$. Then using (\ref{equ:theta}), we infer that $\Theta(\mathcal{F})\simeq \tau^m s^n \Theta(\mathcal{E})$ for some $m, n\in \mathbb{Z}$.

By \cite[Lemma 5.7]{KLM2},  each object $\Theta(\mathcal{E})$ lies in $\mathcal{U}^\mathbb{Z}$. For example, we have
$$\Theta(\mathcal{E}\langle b \vec{x}_3\rangle)=s^{-1}(0, k[t]/(t^{p-b-1}); 0).$$
 In view of \cite[Section 10, p.68]{RZ}, we observe that $\mathcal{U}^\mathbb{Z}$ is closed under $\tau$. It is clear that $\mathcal{U}^\mathbb{Z}$ is closed under degree-shifts. Consequently,  we have $\Theta(\mathcal{V}_2)\subseteq \mathcal{U}^\mathbb{Z}$. On the other hand, each indecomposable object in $\mathcal{U}^\mathbb{Z}$ is of the form $\tau^m s^n(0, k[t]/(t^{p-b-1}); 0)$ for $0\leq b\leq p-2$ and $m, n\in \mathbb{Z}$. Then the equality $\Theta(\mathcal{V}_2)=\mathcal{U}^\mathbb{Z}$ follows immediately. The second statement is a consequence of the equivalence $\Theta$ and (\ref{equ:S-pi}).
\end{proof}

\vskip 10pt

\noindent {\bf Acknowledgements}.\quad  The authors are very grateful to Bernhard Keller for many helpful comments. This work is supported by the National Natural Science Foundation of China (No.s 11671174 and 11671245),  the Fundamental Research Funds for the Central Universities,  and Anhui Initiative in Quantum Information Technologies (AHY150200).

\bibliography{}

\vskip 10pt

 {\footnotesize \noindent Xiao-Wu Chen, Jue Le\\
 Key Laboratory of Wu Wen-Tsun Mathematics, Chinese Academy of Sciences,\\
 School of Mathematical Sciences, University of Science and Technology of China, Hefei 230026, Anhui, PR China}

\end{document}